\documentclass[12pt]{article}
\usepackage[centertags]{amsmath}
\usepackage{amsfonts}
\usepackage{amssymb}
\usepackage{amsthm}
\usepackage{graphicx}
\usepackage{caption}
\usepackage{subcaption}
\input{amssym.def}

\numberwithin{equation}{section}

\newtheorem{Definition}{Definition}[section]
\newtheorem{definition}[Definition]{Definition}
\newtheorem{theorem}[Definition]{Theorem}
\newtheorem{lemma}[Definition]{Lemma}

\newtheorem{corollary}[Definition]{Corollary}

\begin{document}

\title{\Large \bf On the spectrum of linear dependence graph of finite dimensional vector spaces}
\author{\bf A. K. Bhuniya and Sushobhan Maity}
\date{}

\maketitle
\begin{center}
Department of Mathematics, Visva-Bharati, Santiniketan-731235, India. \\
anjankbhuniya@gmail.com, susbhnmaity@gmail.com
\end{center}

\begin{abstract}{\footnotesize }
In this paper, we introduce a graph structure called linear dependence graph of a finite dimensional vector space over a finite field. Some basic properties of the graph like connectedness, completeness, planarity, clique number, chromatic number etc. have been studied. It is shown that two vector spaces are isomorphic if and only if their corresponding linear dependence graphs are isomorphic. Also adjacency spectrum, Laplacian spectrum and distance spectrum of the linear dependence graph have been studied.
\end{abstract}

\section{Introduction}
\pagenumbering{arabic}

There has been a great deal of interest in the last three decades in characterizing different algebraic structures in terms of properties of graphs associated with themselves. There are different formulation on associating a graph with some algebraic structure. Various algebraic structures like semigrups \cite{CGS}, groups \cite{BS,CS}, rings \cite{Beck,CGMS}, etc. have been characterized by assigning graph structures on themselves. Very recently, a series of papers on assigning a graph to a vector space \cite{Das,Das1,Das2,Das3} have been published.

In this paper we define a graph structure on a finite dimensional vector space $V$ over a finite field $F$, called linear dependence graph of the vector space $V$ and study some of its properties. 

In Section 3, we study some basic properties like connectedness, completeness, planarity etc. In Section 4, we study special properties like connectivity, energy etc. by finding out the spectrum of the adjacency matrix, the Laplacian matrix and the distance matrix of the graph.

\section{Definition and Preliminaries}

Let $G=(V,E)$ be a graph. Throughout this paper, every graph is simple. If every pair of distinct vertices of $G$ are adjacent, then it is called complete. A subset $I$ of $V$ is said to be independent if no two elements of $I$ are adjacent. The number of edges of a graph is called the size of $G$. The maximum number of elements of an independent set is called the independence number of $G$. A subset $D$ of $V$ is called dominating if each element of $V\setminus D$ is adjacent to at least one element of $D$. If no proper subset of $D$ is a dominating set for $G$, then $D$ is called a minimal dominating set for $G$. The least cardinality of a dominating set is called the domination number of $G$. A clique is a complete subgraph of $G$. The largest number of a clique is called the clique number of $G$, denoted as $\omega(G)$. The chromatic number of $G$, written as $\chi(G)$, is the minimum number of colours needed for labeling the vertices so that adjacent vertices get different colours. Two graphs $G=(V,E)$ and $G'=(V',E')$ are said to be isomorphic if there exists a bijective mapping $\phi : V \rightarrow V'$ such that $a \thicksim b$ in $G$ if and only if $\phi(a) \thicksim \phi(b)$ in $G'$. A path of length $k$ in a graph $G$ is an alternating sequence of vertices and edges $v_0e_0v_1e_1v_2\cdots v_{k-1}e_{k-1}v_k$, where $v_i$'s are distinct (except possibly $v_0,v_k$) and $e_i$ is the edge joining $v_i$ and $v_{i+1}$. If there exists a path between any pair of distinct vertices, then it is called connected. The distance $d(u,v)$ of two vertices $u,v \in V$, is defined as the length of the shortest path between $u$ and $v$. The diameter of a graph $G$ is defined as $diam(G)=max_{u,v \in V} d(u,v)$, if it exists. Otherwise, $diam(G)$ is defined as $\infty$. A cycle is a path with first and last vertices same. A graph is said to be Eulerian if it contains a cycle consisting of all the edges of $G$ exactly once.

For any graph $G$, let $A(G)$, $D(G)$, $\mathbb{D}(G)$ be the adjacency matrix, diagonal matrix of vertex degrees and the distance matrix respectively. Then the Laplacian matrix of $G$ is defined as $L(G)=D(G)-A(G)$. Interestingly all of these matrices are symmetric. Among them the Laplacian matrix is positive semidefinite and singular with $0$ as the smallest eigenvalue. The eigen values of the adjacency matrix, the Laplacian matrix and the distance matrix are known as the adjacency spectrum, the Laplacian spectrum and the distance spectrum of the graph $G$ respectively. These spectrums plays an important role in the study of many graph theoretic properties like connectivity, colouring, energy of a graph, number of spanning trees etc.

For any square matrix $B$, we denote the characteristic polynomial $det(xI-B)$ of $B$ by $\Theta (B,x)$. For the vertices $\{v_1,v_2, \cdots ,v_n\}$ in $G$, $A_{v_1,v_2, \cdots, v_k}(G)$ is defined as the principal submatrix of $A(G)$ formed by deleting the rows and columns corresponding to the vertices $v_1,v_2, \cdots , v_k$. In particular if $k=|G|$, then for convention it has been taken that $\Theta(A_{v_1,v_2, \cdots ,v_{|G|}}(G),x)=1$.

\begin{definition}
Let $V$ be a finite dimensional vector space over a finite field $F$. Define a graph $\Gamma (V)=(V,E)$, where $V$ is the vector space and $(a,b) \in E$ if and only if $a$ and $b$ are distinct and they are linearly dependent.
\end{definition}

Thus the null vector $\theta$ is adjacent to every other vertices and two non-null vectors $a$ and $b$ are adjacent if and only if $b=\lambda a$ for some $\lambda \in F \setminus \{1\}$.

\begin{figure}[h]
\includegraphics[width=2.0in]{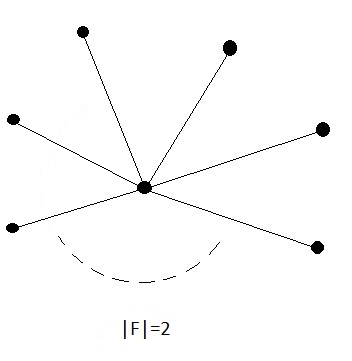} 
\includegraphics[width=2.6in]{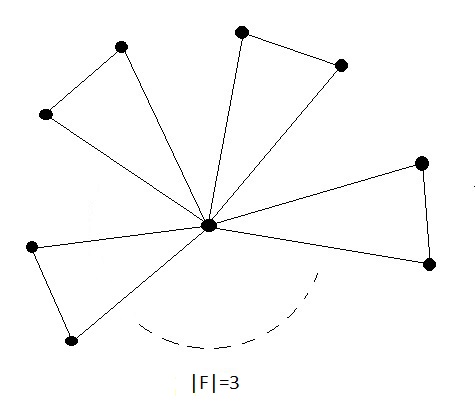}
\includegraphics[width=2.7in]{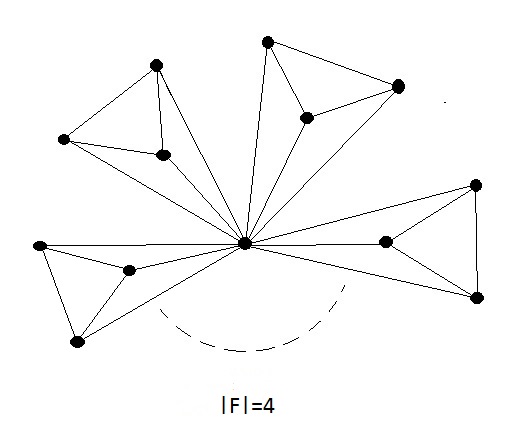}
\caption{Examples in case of $|F|=2,3,4.$}
\end{figure}

Throughout this paper we assume that  $dim(V)=n$ and $|F|=q$.
\section{Basic properties of $\Gamma(V)$}

In this section, we characterize some basic properties of $\Gamma(V)$ like connectedness, completeness, clique number, chromatic number, planarity etc. 

\begin{theorem}    \label{completenes}
$\Gamma (V)$ is complete if and only if $n=1$.
\end{theorem}
\begin{proof}
Let $a\in V$ be a non-zero vector. Since $\Gamma(V)$ is complete for every $b \in V$, $b=\lambda a$ for some $\lambda \in F$. Thus $V=<a>=\{\lambda a: \lambda \in F\}$, showing $dim(V)=1$. 

Converse is trivial.
\end{proof}

\begin{theorem}    \label{size}
The size $m$ of $\Gamma(V)$ is $\frac{q(q^n-1)}{2}$.
\end{theorem}
\begin{proof}
The null vector $\theta$ is adjacent with every non-null vector of $V$. Since $|V|=q^n$, so degree of the vertex $\theta$ is $q^n-1$. Let $a \neq \theta$ be a vertex of $\Gamma(V)$. Then a vertex $b \in \Gamma(V)$ is adjacent with $a$ if and only if $b \neq a$ and $b=\lambda a$ for some $\lambda \in F$. Since $|F|=q$, so degree of each of the non-zero vector is $q-1$. Hence $2m=q^n-1+(q^n-1)(q-1)$ and it follows that $m=\frac{q(q^n-1)}{2}$.
\end{proof}

Since the null vector, $\theta$ is adjacent with every other vertices of $\Gamma(V)$, we have the following result.
\begin{lemma}
$\Gamma(V)$ is connected and $diam(\Gamma(V))=2$.
\end{lemma}
Also it follows that $\{\theta \}$ is the smallest dominating subset of $V$. Thus we have the following result.
\begin{lemma}
The domination number of $\Gamma(V)$ is $1$.
\end{lemma}
Note that the numbers of 1-dimensional subspaces of $V$ are $\frac{q^n-1}{q-1}=q^{n-1}+q^{n-2} + \cdots +1$. We call a non-zero vector of a 1-dimensional subspace $A$, a representative of $A$.

\begin{theorem}
The independence number of $\Gamma(V)$ is $q^{n-1}+q^{n-2}+ \cdots +1$.
\end{theorem}
\begin{proof}
It is easy to see that the subset $I$ of $V$ consists of representatives of the 1-dimensional subspaces of $V$, is a maximal independent set. Then $|I|=q^{n-1}+q^{n-2}+ \cdots +1$. If possible, let there exists an independent set $I'$ of cardinality greater than $q^{n-1}+q^{n-2}+ \cdots +1$. Since total number of 1-dimensional subspaces of $V$ is $q^{n-1}+ q^{n-2}+\cdots +1$, so it follows that there are at least two elements $a,b \in I'$ such that $a,b$ are in a same 1-dimensional subspace of $V$. Then $a \sim b$, which contradicts that $I'$ is independent. Thus the independence number of $\Gamma(V)$ is $q^{n-1}+ q^{n-2}+\cdots +1$.
\end{proof}

\begin{theorem}
Two vector spaces $V$ and $W$ are isomorphic if and only if their corresponding graphs $\Gamma(V)$ and $\Gamma(W)$ are isomorphic.
\end{theorem}
\begin{proof}
Let $V$ and $W$ be isomorphic. Then there exists a vector space isomorphism $T:V \rightarrow W$. If $\alpha \sim \beta$ in $\Gamma(V)$, then $\alpha =\lambda \beta$ for some $ \lambda \in F$ and so $T(\alpha)=\lambda T(\beta)$. Thus $T(\alpha) \sim T(\beta)$ in $\Gamma(W)$. Since $T$ is an isomorphism, similarly $T(\alpha) \sim T(\beta)$ in $\Gamma(W)$ implies that $\alpha \sim \beta$ in $\Gamma(V)$. Hence $\Gamma(V)$ and $\Gamma(W)$ are isomorphic as graphs.

Conversely, let $\phi:\Gamma(V) \rightarrow \Gamma(W)$ be a graph isomorphism. Let $dim V=m$ and $dim W=n$. Since the independence numbers of two isomorphic graphs are the same, so we have $q^{m-1}+q^{m-2}+ \cdots +1 = q^{n-1}+q^{n-2}+ \cdots +1$ and so $m=n$. Hence the vector spaces $V$ and $W$ are isomorphic. 
\end{proof}

\begin{theorem}
Let $M$ be a clique of $\Gamma(V)$, then $M$ is maximal if and only if $M$ is an 1-dimensional subspace of $V$. Hence the clique number $\omega(\Gamma(V))$ of $\Gamma(V)$ is $q$.
\end{theorem}
\begin{proof}
Let $M$ be a maximal clique of $\Gamma(V)$. Since $\theta$ is adjacent with every other vertices of $\Gamma(V)$, there exists $x \neq  \{\theta\}$ such that $\theta , x \in M$. Since $M$ is maximal, $ <x> =\{\lambda x| \lambda \in F\} \subseteq M$. Now, let $z \in M$. Then $z \sim x$ implies that $z=\lambda x$ for some $\lambda \in F$ and so $z \in <x>$. Hence $M=<x>$ for some $x \in V \setminus \{\theta\}$, showing that $M$ is an 1-dimensional subspace of $V$.

Converse is trivial.
\end{proof}

\begin{theorem}
Chromatic number of $\Gamma(V)$ is $q$.
\end{theorem}
\begin{proof}
If two vectors $a$ and $b$ are adjacent in $\Gamma(V)$, then they are in a same 1-dimensional subspace of $V$. Now $|F|=q$ implies that an 1-dimensional subspace of $V$ contains $q$ elements. Hence the chromatic number $\chi(\Gamma(V)) \leq q$. Also $\chi(\Gamma(V)) \geq \omega(\Gamma(V))=q$. Hence $\chi(\Gamma(V))=q$.
\end{proof}

\begin{lemma}
If $q$ is odd, then $\Gamma(V)$ is Eulerian.
\end{lemma}
\begin{proof}
If $q$ is an odd positive integer, then from the proof of the Theorem \ref{size}, we see that the degree of every vertex of $\Gamma(V)$ is an even integer. Hence the graph $\Gamma(V)$ is Eulerian.
\end{proof}

\begin{lemma}
Edge connectivity of $\Gamma(V)$ is $q-1$.
\end{lemma}
\begin{proof}
We have diameter of $\Gamma(V)$ is $2$. So, by  Theorem 3.1 \cite{PL}, its edge connectivity is equal to its minimum degree, i.e. $q-1$.
\end{proof}

Deletion of the vertex $\theta$ makes the graph $\Gamma(V)$ disconnected. Hence we have the following result.
\begin{lemma}
The vertex connectivity of the graph $\Gamma(V)$ is $1$.
\end{lemma}

\begin{lemma}
$\Gamma(V)$ is planar if and only if $q=2,3,4$.
\end{lemma}
\begin{proof}
Every 1-dimensional subspace of $V$ is a complete subgraph of $\Gamma(V)$. So, for $q \geq 5$, $\Gamma(V)$ contains $K_5$, the complete subgraph of $\Gamma(V)$ of five vertices, which is nonplanar. Rest of the result follows from Figure 1.
\end{proof}

\section{Spectrums of the graph $\Gamma(V)$}

In this section, we determine the eigen values of the adjacency matrix, the Laplacian matrix and the distance matrix of the graph $\Gamma(V)$.

We note that the elements of an 1-dimensional subspace of $V$ are adjacent with each other. So, the adjacency matrix $A(\Gamma(V))$ of $\Gamma(V)$, is given below, where rows and columns are indexed by the vertices  $\theta$ and the non-zero elements of all 1-dimensional subspaces successively. Thus

\begin{align*}
&A(\Gamma(V))=\left(
\begin{array}{c|cccc|cccc|c|cccc}
0 &1 &1 &\cdots &1 &1 &1 &\cdots &1 &\cdots &1 &1 &\cdots &1\\
\hline
1 &0 & 1 &\cdots &1 &0 &0  & \cdots &0 &\cdots &0 &0&\cdots &0\\
1 &1 & 0 & \cdots &1 &0 &0  & \cdots &0 &\cdots &0 &0& \cdots &0\\
\vdots &\vdots &\vdots &\ddots  &\vdots &\vdots &\vdots &\ddots  &\vdots &\vdots &\vdots & \vdots &\ddots &\vdots\\
1 &1 &1 & \cdots &0 & 0 &0 & \cdots &0 &\cdots&0 &0 & \cdots &0\\
\hline
1 &0 &0 & \cdots &0 & 0 & 1 &\cdots &1 &\cdots &0 &0 & \cdots &0\\
1 &0 &0 & \cdots &0 &1 & 0 & \cdots &1 &\cdots &0  &0 &\cdots &0\\
\vdots &\vdots &\vdots &\ddots  &\vdots &\vdots &\vdots &\ddots  &\vdots &\vdots &\vdots &\vdots &\ddots &\vdots\\
1 &0 & 0 & \cdots &0 &1 &1 & \cdots &0  &\cdots &0 &0 & \cdots &0\\
\hline
\vdots &\vdots &\vdots &\vdots  &\vdots &\vdots &\vdots &\vdots  &\vdots &\ddots &\vdots & \vdots &\vdots &\vdots\\
\hline
1 &0 &0 & \cdots &0 & 0 & 0 &\cdots &0 &\cdots &0 &1 & \cdots &1\\
1 &0 &0 & \cdots &0 &0 & 0 & \cdots &0 &\cdots &1  &0 &\cdots &1\\
\vdots &\vdots &\vdots &\ddots  &\vdots &\vdots &\vdots &\ddots  &\vdots &\vdots &\vdots &\vdots &\ddots &\vdots\\
1 &0 & 0 & \cdots &0 &0 &0 & \cdots &0  &\cdots &1 &1 & \cdots &0
\end{array}
\right )
\end{align*}

Order of this matrix is $q^n$. Hence $A_{\theta}(\Gamma(V))$ is a matrix of order $q^n-1$. Also $A_{\theta}(\Gamma(V))$ is a block diagonal matrix, having $q^{n-1}+q^{n-2}+ \cdots +q+1$ blocks each of order $q-1$ and entries on the principal diagonal are $0$ and $1$ otherwise. Also, the degree of the null vector is $q^n-1$ and the degree of all other elements is $q-1$. So the diagonal matrix of vertex degrees of $\Gamma(V)$ is 
\begin{center}
$D(\Gamma(V))=\left(
\begin{array}{c|cccc}
q^n-1 &0 &0 &\cdots &0\\
\hline
0 &q-1 &0 &\cdots &0\\
0 &0 &q-1 &\cdots &0\\
\vdots &\vdots &\vdots &\ddots &\vdots\\
0 &0 &0 &\cdots &q-1
\end{array}
\right )$
\end{center} 

Then the Laplacian matrix of $\Gamma(V)$ is 
\begin{center}
$L(\Gamma(V))=\left(
\begin{array}{c|cccc|c|cccc}
q^n-1 &-1 &-1 &\cdots &-1 &\cdots &-1 &-1 &\cdots &-1\\
\hline
-1 &q-1 & -1 &\cdots &-1  &\cdots &0 &0&\cdots &0\\
-1 &-1 & q-1 & \cdots &-1 &\cdots &0 &0& \cdots &0\\
\vdots &\vdots &\vdots &\ddots  &\vdots &\vdots &\vdots & \vdots &\ddots &\vdots\\
-1 &-1 &-1 & \cdots &q-1 &\cdots&0 &0 & \cdots &0\\
\hline
\vdots &\vdots &\vdots &\vdots  &\vdots &\ddots &\vdots & \vdots &\vdots &\vdots\\
\hline
-1 &0 &0 & \cdots &0 &\cdots &q-1 &-1 & \cdots &-1\\
-1 &0 &0 & \cdots &0 &\cdots &-1  &q-1 &\cdots &-1\\
\vdots &\vdots &\vdots &\ddots  &\vdots  &\vdots &\vdots &\vdots &\ddots &\vdots\\
-1 &0 & 0 & \cdots &0 &\cdots &-1 &-1 & \cdots &q-1
\end{array}
\right )$
\end{center}

The distance between any two vertices of an 1-dimensional subspace is 1 and is 2 if they belong to two distinct 1-dimensional subspaces of $V$. So the distance matrix of $\Gamma(V)$ is 
\begin{center}
$\mathbb{D}(V)=\left(
\begin{array}{c|cccc|cccc|c|cccc}
0 &1 &1 &\cdots &1 &1 &1 &\cdots &1 &\cdots &1 &1 &\cdots &1\\
\hline
1 &0 & 1 &\cdots &1 &2 &2  & \cdots &2 &\cdots &2 &2 &\cdots &2\\
1 &1 & 0 & \cdots &1 &2 &2  & \cdots &2 &\cdots &2 &2& \cdots &2\\
\vdots &\vdots &\vdots &\ddots  &\vdots &\vdots &\vdots &\ddots  &\vdots &\vdots &\vdots & \vdots &\ddots &\vdots\\
1 &1 &1 & \cdots &0 & 2 &2 & \cdots &2 &\cdots&2 &2 & \cdots &2\\
\hline
1 &2 &2 & \cdots &2 & 0 & 1 &\cdots &1 &\cdots &2 &2 & \cdots &2\\
1 &2 &2 & \cdots &2 &1 & 0 & \cdots &1 &\cdots &2  &2 &\cdots &2\\
\vdots &\vdots &\vdots &\ddots  &\vdots &\vdots &\vdots &\ddots  &\vdots &\vdots &\vdots &\vdots &\ddots &\vdots\\
1 &2 & 2 & \cdots &2 &1 &1 & \cdots &0  &\cdots &2 &2 & \cdots &2\\
\hline
\vdots &\vdots &\vdots &\vdots  &\vdots &\vdots &\vdots &\vdots  &\vdots &\ddots &\vdots & \vdots &\vdots &\vdots\\
\hline
1 &2 &2 & \cdots &2 & 2 & 2 &\cdots &2 &\cdots &0 &1 & \cdots &1\\
1 &2 &2 & \cdots &2 &2 & 2 & \cdots &2 &\cdots &1  &0 &\cdots &1\\
\vdots &\vdots &\vdots &\ddots  &\vdots &\vdots &\vdots &\ddots  &\vdots &\vdots &\vdots &\vdots &\ddots &\vdots\\
1 &2 &2 & \cdots &2 &2 &2 & \cdots &2  &\cdots &1 &1 & \cdots &0
\end{array}
\right )$
\end{center}
Hence we have the following results.
\begin{theorem}      \label{adjacencypoly}
The characteristic polynomial of the adjacency matrix of $\Gamma(V)$ is 
\begin{center}
$\Theta(A(\Gamma(V)),x)=\{x^2-(q-2)x-(q^n-1)\}\{x-(q-2)\}^{q^{n-1}+ \cdots +q} (x+1)^{(q-2)(q^{n-1}+ \cdots +1)}$.
\end{center}
\end{theorem}
\begin{proof}
The characteristic polynomial of $A(\Gamma(V))$ is
\begin{center}
$\Theta(A(\Gamma(V)),x)=
\begin{array}{|cccccccccc|}
x &-1 &-1 &\cdots &-1 &\cdots &-1 &-1 &\cdots &-1\\
-1 &x & -1 &\cdots &-1 &\cdots &0 &0&\cdots &0\\
-1 &-1 & x & \cdots &-1 &\cdots &0 &0& \cdots &0\\
\vdots &\vdots &\vdots &\ddots  &\vdots &\vdots &\vdots & \vdots &\vdots &\vdots\\
-1 &-1 & -1 & \cdots &x &\cdots &0 &0& \cdots &0\\
\vdots &\vdots &\vdots &\vdots  &\vdots &\ddots &\vdots & \vdots &\vdots &\vdots\\
-1 &0 &0 & \cdots &0 &\cdots &x &-1 & \cdots &-1\\
-1 &0 &0 & \cdots &0 &\cdots &-1  &x &\cdots &-1\\
\vdots &\vdots &\vdots &\ddots  &\vdots &\vdots &\vdots &\vdots &\ddots &\vdots\\
-1 &0 & 0 & \cdots &0   &\cdots &-1 &-1 & \cdots &x
\end{array}
$
\end{center}

Multiply the first row by $(x-(q-2))$ and apply the row operation $R_1'=R_1+R_2+ \cdots+R_{q^n}$. Then expanding the determinant in terms of the first row, we get
\begin{center}
$\Theta (A(\Gamma(V)),x)=\frac{\{x^2-(q-2)x-(q^n-1)\}}{x-(q-2)}\cdot
\begin{array}{|cccc|}
x & -1 &\cdots &-1 \\
-1 & x & \cdots &-1 \\
\vdots &\vdots &\ddots  &\vdots\\
-1 &-1 & \cdots &x
\end{array}^{q^{n-1}+ q^{n-2}+ \cdots +q+1}_{(q-1) \times (q-1)}$.
\end{center}

Let 
\begin{equation}    \label{adjacency}
A_1=\begin{array}{|cccc|}
x & -1 &\cdots &-1 \\
-1 & x & \cdots &-1 \\
\vdots &\vdots &\ddots  &\vdots\\
-1 &-1 & \cdots &x
\end{array}_{(q-1)\times (q-1)}
\end{equation}

Multiply the first row of \ref{adjacency}, by $(x-(q-3))$ and apply the row operation by $R_1'=R_1+R_2+ \cdots+R_{q-1}$. Then expanding the determinant in terms of the first row, we get
\begin{center}
$A_1=\frac{(x+1)(x-(q-2))}{x-(q-3)}\cdot
\begin{array}{|cccc|}
x & -1 &\cdots &-1 \\
-1 & x & \cdots &-1 \\
\vdots &\vdots &\ddots  &\vdots\\
-1 &-1 & \cdots &x
\end{array}_{(q-2)\times (q-2)}$
\end{center}

Again multiply the 1 st row by $(x-(q-4))$ and then apply the row operation $R_1'=R_1+R_2+\cdots +R_{q-2}$. Then expanding in terms of 1 st row, we get
\begin{center}
$A_1=\frac{(x+1)^2(x-(q-2))}{x-(q-4)}\cdot
\begin{array}{|cccc|}
x & -1 &\cdots &-1 \\
-1 & x & \cdots &-1 \\
\vdots &\vdots &\ddots  &\vdots\\
-1 &-1 & \cdots &x
\end{array}_{(q-3)\times (q-3)}$
\end{center}

Continuing in this way, we get $A_1=(x+1)^{q-2}(x-(q-2))$ and so, 
\begin{align*}
\Theta(A(\Gamma(V)),x)=&\frac{\{x^2-(q-2)x-(q^n-1)\}}{\{x-(q-2)\}}\{(x-(q-2)\}^{q^{n-1}+ \cdots +1} (x+1)^{(q-2)(q^{n-1}+ \cdots +1)}\\
=&\{x^2-(q-2)x-(q^n-1)\}\{x-(q-2)\}^{q^{n-1}+ \cdots +q} (x+1)^{(q-2)(q^{n-1}+ \cdots +1)}.
\end{align*}
\end{proof}

If $\lambda_1, \lambda_2, \cdots , \lambda_n$ are the adjacency eigen values of a graph $\Gamma$, then the energy of the graph, denoted by $\varepsilon (\Gamma)$, is defined to be $\sum^n_{i=1}|\lambda_i|$. Thus from Theorem \ref{adjacencypoly}, it follows that
\begin{corollary}
The energy of the graph $\Gamma(V)$ is $2(q-2)(q^{n-1}+ \cdots +1)$.
\end{corollary}

\begin{theorem}       \label{laplacianpoly}
The characteristic polynomial of the Laplacian matrix of $\Gamma(V)$ is
\begin{center}
$\Theta (L(\Gamma(V)),x)=x(x-q^n)(x-1)^{q^{n-1}+ \cdots +q}(x-q)^{(q-2)(q^{n-1}+ \cdots +1)}$.
\end{center}
\end{theorem}
\begin{proof}
The characteristic polynomial of $L(\Gamma(V))$ is
\begin{center}
$\Theta(L(\Gamma(V)),x)=
\begin{array}{|cccccccc|}
x-(q^n-1) &1 &\cdots &1  &\cdots &1 &\cdots &1\\
1 &x-(q-1)  &\cdots &1 &\cdots &0 & \cdots &0\\
\vdots &\vdots  &\ddots  &\vdots &\vdots &\vdots &\ddots &\vdots\\
1 &1  &\cdots &x-(q-1)  &\cdots &0 & \cdots &0\\
\vdots &\vdots &\vdots  &\vdots &\ddots & \vdots &\vdots &\vdots\\
1 &0 & \cdots &0 &\cdots &x-(q-1) & \cdots &1\\
\vdots &\vdots &\ddots  &\vdots &\vdots &\vdots  &\ddots &\vdots\\
1 &0 & \cdots &0  &\cdots &1 & \cdots &x-(q-1)
\end{array}$
\end{center}

Applying the row operation $R_1'=R_1+R_2+\cdots R_{q^n}$, we get
\begin{center}
$\Theta(L(\Gamma(V)),x)=
x \cdot
\begin{array}{|cccccccc|}
1 &1 &\cdots &1  &\cdots &1 &\cdots &1\\
1 &x-(q-1)  &\cdots &1 &\cdots &0 & \cdots &0\\
\vdots &\vdots  &\ddots  &\vdots &\vdots &\vdots &\ddots &\vdots\\
1 &1  &\cdots &x-(q-1)  &\cdots &0 & \cdots &0\\
\vdots &\vdots &\vdots  &\vdots &\ddots & \vdots &\vdots &\vdots\\
1 &0 & \cdots &0 &\cdots &x-(q-1) & \cdots &1\\
\vdots &\vdots &\ddots  &\vdots &\vdots &\vdots  &\vdots &\vdots\\
1 &0 & \cdots &0  &\cdots &1 & \cdots &x-(q-1)
\end{array}$
\end{center}

Multiply the first row by $(x-1)$ and then apply the row operation $R_1'=R_1-R_2- \cdots-R_{q^n}$. Then expanding the determinant in terms of the first row, we get
\begin{center}
$\Theta (L(\Gamma(V)),x)=\frac{x(x-q^n)}{(x-1)}\cdot
\begin{array}{|cccc|}
x-(q-1) & 1 &\cdots &1 \\
1 & x-(q-1) & \cdots &1 \\
\vdots &\vdots &\ddots  &\vdots\\
1 &1 & \cdots &x-(q-1)
\end{array}^{q^{n-1}+q^{n-2}+\cdots +q+1}_{(q-1) \times (q-1)}$
\end{center}

Let 
\begin{equation}    \label{laplacian}
L_1=\begin{array}{|cccc|}
x-(q-1) & 1 &\cdots &1 \\
1 & x-(q-1) & \cdots &1 \\
\vdots &\vdots &\ddots  &\vdots\\
1 &1 & \cdots &x-(q-1)
\end{array}_{(q-1)\times (q-1)}
\end{equation}

Multiply the first row of \ref{laplacian}, by $(x-2)$ and apply the row operation by $R_1'=R_1-R_2- \cdots-R_{q-1}$. Then expanding the determinant in terms of the first row, we get
\begin{center}
$L_1=\frac{(x-1)(x-q)}{(x-2)}\cdot 
\begin{array}{|cccc|}
x-(q-1) & 1 &\cdots &1 \\
1 & x-(q-1) & \cdots &1 \\
\vdots &\vdots &\ddots  &\vdots\\
1 &1 & \cdots &x-(q-1)
\end{array}_{(q-2)\times (q-2)}$
\end{center}

Again multiply the 1 st row by $(x-3)$ and then apply the row operation $R_1'=R_1-R_2-\cdots -R_{q-2}$. Then expanding in terms of 1 st row, we get
\begin{center}
$L_1=\frac{(x-1)(x-q)^2}{(x-3)}\cdot 
\begin{array}{|cccc|}
x-(q-1) & 1 &\cdots &1 \\
1 & x-(q-1) & \cdots &1 \\
\vdots &\vdots &\ddots  &\vdots\\
1 &1 & \cdots &x-(q-1)
\end{array}_{(q-3)\times (q-3)}$
\end{center}

Continuing in this way, we get $L_1=(x-1)(x-q)^{q-2}$.

So, 
\begin{align*}
\Theta(L(\Gamma(V)),x)=&\frac{x(x-q^n)}{(x-1)}\cdot \{(x-1)(x-q)^{q-2}\}^{(q^{n-1}+ \cdots +1)}\\
=&x(x-q^n)(x-1)^{q^{n-1}+ \cdots +q}(x-q)^{(q-2)(q^{n-1}+ \cdots +1)}.
\end{align*}
\end{proof}

The second smallest eigen value of the Laplacian matrix $L(\Gamma)$ of a graph $\Gamma$, is denoted by $a(\Gamma)$. This quantity shares many properties with the vertex or edge-connectivity and according to Fiedler \cite{Fiedler}, is called the algebraic connectivity of $\Gamma$. So, from Theorem \ref{laplacianpoly}, we have the following result. 
\begin{corollary}
The algebraic connectivity of $\Gamma(V)$ is $1$.
\end{corollary}

If $\lambda_1\geq \lambda_2 \geq \cdots \geq \lambda_n=0$ are the Laplacian eigen values of the graph $\Gamma$ of n vertices, then the number of spanning trees of $\Gamma$, denoted by $\tau(\Gamma)$, is $\frac{\lambda_1\lambda_2 \cdots \lambda_{n-1}}{n}$ [Theorem 4.11; \cite{Bapat}]. Thus from Theorem \ref{laplacianpoly}, we get
\begin{corollary}
The number of spanning trees of $\Gamma(V)$ is $q^{(q-2)(q^{n-1}+\cdots +1)}$.
\end{corollary}

Recently Gutman et. al \cite{GZ} have defined the Laplacian energy of a graph $\Gamma$ with n vertices and $m$ edges as: $LE(\Gamma)=\sum^n_{i=1}|\lambda_i-\frac{2m}{n}|$, where $\lambda_i$ are the Laplacian eigen values of the graph $\Gamma$. From Theorem \ref{size} and Theorem \ref{laplacianpoly}, we have:

\begin{corollary}
The Laplacian energy of the graph $\Gamma(V)$ is $q^n+(\frac{q^n(q-1)-q}{q^n})(q^{n-1}+\cdots +q)+(\frac{q}{q^n})(q-2)(q^{n-1}+ \cdots +1)$.
\end{corollary}

\begin{theorem}              \label{distancepoly}
The characteristic polynomial of the distance matrix of $\Gamma(V)$ is
\begin{center}
$\Theta (\mathbb{D}(\Gamma(V)),x)=[x^2-\{2(q^n-1)-q\}x-(q^n-1)](x+q)^{q^{n-1}+ \cdots +q}(x+1)^{(q-2)(q^{n-1}+ \cdots +1)}$.
\end{center}
\end{theorem}
\begin{proof}
The characteristic polynomial of $\mathbb{D}(\Gamma(V))$ is
\begin{center}
$\Theta(\mathbb{D}(\Gamma(V)),x)=
\begin{array}{|cccccccccc|}
x &-1 &-1 &\cdots &-1 &\cdots &-1 &-1 &\cdots &-1\\
-1 &x & -1 &\cdots &-1 &\cdots &-2 &-2 &\cdots &-2\\
-1 &-1 & x & \cdots &-1 &\cdots &-2 &-2& \cdots &-2\\
\vdots &\vdots &\vdots &\ddots  &\vdots &\vdots &\vdots & \vdots &\ddots &\vdots\\
-1 &-1 &-1 & \cdots &x &\cdots &-2 &-2 & \cdots &-2\\
\vdots &\vdots &\vdots &\vdots &\vdots &\ddots &\vdots & \vdots &\vdots &\vdots\\
-1 &-2 &-2 & \cdots &-2 &\cdots &x &-1 & \cdots &-1\\
-1 &-2 &-2 & \cdots &-2 &\cdots &-1  &x &\cdots &-1\\
\vdots &\vdots &\vdots &\ddots  &\vdots &\vdots &\vdots &\vdots &\ddots &\vdots\\
-1 &-2 &-2 & \cdots &-2 &\cdots &-1 &-1 & \cdots &x
\end{array}
$
\end{center}

Apply the successive column operations $C_i'=C_i-2C_1$ for $i=2,3, \cdots ,q^n$; we get

\begin{align*}
\Theta(\mathbb{D}(\Gamma(V)),x) \; =
& \; \begin{array}{|cccccccc|}
x &-1-2x &\cdots &-1-2x  &\cdots &-1-2x &\cdots &-1-2x\\
-1 &x+2  &\cdots &1 &\cdots &0 & \cdots &0\\
\vdots &\vdots  &\ddots  &\vdots &\vdots &\vdots &\ddots &\vdots\\
-1 &1  &\cdots &x+2  &\cdots &0 & \cdots &0\\
\vdots &\vdots &\vdots  &\vdots &\ddots & \vdots &\vdots &\vdots\\
-1 &0 & \cdots &0 &\cdots &x+2 & \cdots &1\\
\vdots &\vdots &\ddots  &\vdots &\vdots &\vdots  &\ddots &\vdots\\
-1 &0 & \cdots &0  &\cdots &1 & \cdots &x+2
\end{array}  
\\
=& \; (1+2x)\cdot
\begin{array}{|cccccccc|}
\frac{x}{1+2x} &-1 &\cdots &-1  &\cdots &-1 &\cdots &-1\\
-1 &x+2  &\cdots &1 &\cdots &0 & \cdots &0\\
\vdots &\vdots  &\ddots  &\vdots &\vdots &\vdots &\ddots &\vdots\\
-1 &1  &\cdots &x+2  &\cdots &0 & \cdots &0\\
\vdots &\vdots &\vdots  &\vdots &\ddots & \vdots &\vdots &\vdots\\
-1 &0 & \cdots &0 &\cdots &x+2 & \cdots &1\\
\vdots &\vdots &\ddots  &\vdots &\vdots &\vdots  &\vdots &\vdots\\
-1 &0 & \cdots &0  &\cdots &1 & \cdots &x+2
\end{array}
\end{align*}

Multiply the 1 st row by $x+q$ and then apply the row operation $R_1'=R_1+ \cdots +R_{q^n}$, we get 
\begin{center}
$\Theta(\mathbb{D}(\Gamma(V)),x)=
\frac{[x^2-\{2(q^n-1)-q\}x-(q^n-1)]}{(x+q)}\cdot 
\begin{array}{|cccc|}
x+2 & 1 &\cdots &1 \\
1 & x+2 & \cdots &1 \\
\vdots &\vdots &\ddots  &\vdots\\
1 &1 & \cdots &x+2
\end{array}^{q^{n-1}+q^{n-2}+\cdots +q+1}_{(q-1)\times (q-1)}$
\end{center}

Let 
\begin{equation}     \label{distancity}
D_1=
\begin{array}{|cccc|}
x+2 & 1 &\cdots &1 \\
1 & x+2 & \cdots &1 \\
\vdots &\vdots &\ddots  &\vdots\\
1 &1 & \cdots &x+2
\end{array}_{(q-1)\times (q-1)}
\end{equation}
\\Multiply the first row of \ref{distancity}, by $(x+(q-1))$ and apply the row operation $R_1'=R_1-R_2- \cdots -R_{q-1}$. Then expanding the determinant in terms of the first row, we get
\begin{center}
$D_1=\frac{(x+q)(x+1)}{(x+(q-1))}\cdot 
\begin{array}{|cccc|}
x+2 & 1 &\cdots &1 \\
1 & x+2 & \cdots &1 \\
\vdots &\vdots &\ddots  &\vdots\\
1 &1 & \cdots &x+2
\end{array}_{(q-2)\times (q-2)}$
\end{center}

Again multiply the 1 st row by $(x+(q-2))$ and then apply the row operation $R_1'=R_1-R_2-\cdots -R_{q-2}$. Then expanding in terms of 1 st row, we get
\begin{center}
$D_1=\frac{(x+q)(x+1)^2}{(x+(q-2))}\cdot 
\begin{array}{|cccc|}
x+2 & 1 &\cdots &1 \\
1 & x+2 & \cdots &1 \\
\vdots &\vdots &\ddots  &\vdots\\
1 &1 & \cdots &x+2
\end{array}_{(q-3)\times (q-3)}$
\end{center}

Continuing in this way, we get $D_1=(x+q)(x+1)^{q-2}$ and so, 
\begin{align*}
\Theta(\mathbb{D}(\Gamma(V)),x)=&\frac{x^2-\{2(q^n-1)-q\}x-(q^n-1)}{(x+q)}\{(x+q)(x+1)^{q-2})^{(q^{n-1}+ \cdots +1)}\\
=&[x^2-\{2(q^n-1)-q\}x-(q^n-1)](x+q)^{q^{n-1}+ \cdots +q} (x+1)^{(q-2)(q^{n-1}+ \cdots +1)}.
\end{align*}
\end{proof}

Recently Indulal, Gutman and Vijayakumar \cite{IGV} have defined the distance energy of a graph $\Gamma$ as: $E_{\mathbb{D}}(\Gamma)=\sum^n_{i=1}|\lambda_i|$, where $\lambda_i$ are the distance eigen values of the graph $\Gamma$. From Theorem \ref{distancepoly}, we have:

\begin{corollary}
The distance  energy of the graph $\Gamma(V)$ is $2(2q^n-q-2)$.
\end{corollary}

\bibliographystyle{amsplain}

\begin{thebibliography}{10}
\baselineskip 5mm

\bibitem{Bapat}
R. B. Bapat: Graphs and matrices. Second edition, Hindustan Book Agency. 2014.

\bibitem{Beck}
I. Beck: Coloring of commutative rings. Journal of Algebra. 116 (1988), 208-226.

\bibitem{BS}
A. K. Bhuniya and S. Bera: On some characterizations of strong power graphs of finite groups. Spec. Matrices. 4 (2016), 121-129.

\bibitem{CS}
P. J. Cameron and S. Ghosh: The power graph of a finite group. Discrete mathematics. 311(13) (2011), 1220-122.

\bibitem{CGS}
I. Chakrabarty, S. Ghosh and M. K. Sen: Undirected power graph of semigroups. Semigroup Forum. 78 (2009), 410-426.

\bibitem{CGMS}
I. Chakrabarty, S. Ghosh, T. K. Mukherjee and M. K. Sen: Intersection graphs of ideals of rings. Discrete mathematics 309. 17 (2009), 5381-5392.
 
\bibitem{Das}
A. Das: Non-Zero component graph of a finite dimensional vector spaces. Communications in Algebra. 44 (2016), 3918-3926.

\bibitem{Das1}
A. Das: Non-Zero component union graph of a finite dimensional vector space. Linear and Multilinear Algebra. DOI: 10.1080/03081087.2016.1234577.

\bibitem{Das2}
A. Das: Subspace inclusion graph of a vector space. Communications in Algebra. 44 (2016), 4724-4731.

\bibitem{Das3}
A. Das: On non-zero component graph of vector spaces over finite fields. J. Algebra Appl. 16(1) (2017), DOI: 10.1142/S0219498817500074.

\bibitem{Fiedler}
M. Fiedler:  Algebraic connectivity of graphs. Czechoslovak Math. J. 23 (1973), 298-305.

\bibitem{GZ}
I. Gutman I. and B. Zhou: Laplacian energy of a graph. Linear Algebra Appl. 414 (2006), 29-37.

\bibitem{IGV}
G. Indulal, I. Gutman and A. Vijayakumar: On distance energy of graphs. MATCH Commun. Mathe. Comput. Chem. 60 (2008), 461-472.

\bibitem{PL}
J. Plensik: Critical graphs of given diameter. Acta Fac. Rerum Natur. Univ. Comenian. Math. 30 (1975), 71-93.


\end{thebibliography}

\end{document}